\documentclass[12pt,psamsfonts]{amsart}
\usepackage{amssymb,eucal}
\usepackage[applemac]{inputenc}
\usepackage[T1]{fontenc}
\usepackage{yfonts}
\usepackage{graphicx}
\usepackage{epstopdf}
\usepackage{enumerate}

\def\N{\mathbb N}
\def\Z{\mathbb Z}

\theoremstyle{plain}
\newtheorem{theorem}{Theorem}

\newtheorem*{lemma}{Lemma}

\theoremstyle{definition}

\theoremstyle{remark}

\begin{document}

\title{Wolstenholme again}
\author{Christian Aebi and Grant Cairns}
\address{Coll\`ege Calvin, Geneva, Switzerland 1211}
\email{christian.aebi@edu.ge.ch}
\address{La Trobe University, Melbourne, Australia 3086}
\email{G.Cairns@latrobe.edu.au}

\maketitle

Thanks to Wolstenholme \cite{WJ},  the following three congruences have been know since 1862, for all primes $p\ge 5$:
\begin{align}
\binom{2p-1}{p-1}&\equiv 1 \pmod{p^3},\label{e3}\\
1+\frac{1}{2}+\frac{1}{3}+\ldots+\frac{1}{p-1} &\equiv 0 \pmod{p^2},\label{e2}\\
1+\frac{1}{2^2}+\frac{1}{3^2}+\ldots+\frac{1}{(p-1)^2}&\equiv 0 \pmod{p}.\label{e1}
\end{align}
Here, of course, $\frac{1}{k}$ means the (multiplicative) inverse of $k$ in the relevant sense: in $\Z_p$,  $\Z_{p^2}$, etc, according to the context.
  More than 125 years later, Gardiner \cite{GA} showed the relation between these equivalences when the degree is pushed one level higher:
\begin{theorem}\label{T:mod4}
If $p\ge 7$ is prime, the following conditions are equivalent:
\begin{enumerate}[{\rm (a)}]
  \item $p$ is a Wolstenholme prime, meaning : $\binom{2p-1}{p-1}\equiv 1 \pmod{p^4}$,
  \item $1+\frac{1}{2}+\frac{1}{3}+\ldots+\frac{1}{p-1} \equiv 0 \pmod{p^3}$,
  \item $1+\frac{1}{2^2}+\frac{1}{3^2}+\ldots+\frac{1}{(p-1)^2} \equiv 0 \pmod{p^2}$,
  \item $p\mid B_{p-3}$ where $B_{k}$ denotes the $k^\text{th}$ Bernoulli number. 
\end{enumerate}
\end{theorem}
The term \textit{Wolstenholme prime} was introduced by McIntosh in \cite{Mc}. The only known Wolstenholme primes are 16843 and 2124679; see \cite{MeIS} for another equivalent condition. Nevertheless,  Gardiner's result has been extended one degree further.

\begin{theorem}\label{T:mod5}
If $p\ge 7$ is prime, the following conditions are equivalent:
\begin{enumerate}[{\rm (a)}]
  \item $\binom{2p-1}{p-1}\equiv 1 \pmod{p^5}$,
  \item $1+\frac{1}{2}+\frac{1}{3}+\ldots+\frac{1}{p-1} \equiv 0 \pmod{p^4}$,
  \item $1+\frac{1}{2^2}+\frac{1}{3^2}+\ldots+\frac{1}{(p-1)^2} \equiv 0 \pmod{p^3}$,
  \item $p^2 \mid B_{p^3-p^2-2}$.
\end{enumerate}
\end{theorem}

The above result is implicitly contained in Helou and Terjanian's 2008 paper \cite{HT}, but somewhat scattered amongst a raft of other, often more substantial results.  We will say more on this at the end of this note. Our main goal here is to highlight the result itself, and to provide a unified, elementary and direct proof. One basic classical result we  use freely throughout this note was proved by Leudesdorf \cite{Leud}; see also \cite[Chap.~VIII.8.7]{HW} and \cite{Ge}:

\begin{lemma}
If $p\ge 7$ is prime and $k\in \N$ such that $2k<p-1$, then
\begin{align}
\sum_{1 \le i \le p-1}\frac{1}{i^{2k-1}}\equiv 0 \pmod{p^2},\ \text{and}\label{ig2}
\\
\sum_{1 \le i \le p-1}\frac{1}{i^{2k}}\equiv 0 \pmod{p}.\label{ig1}
\end{align}
\end{lemma}

\begin{proof}[Proof of Theorem \ref{T:mod5}] (a) $\Leftrightarrow$ (c).
We first develop the binomial coefficient $\binom{2p-1}{p-1}$ ``downwards'':
\begin{align*}
 \binom{2p-1}{p-1}   &= \frac{(2p-1)(2p-2)\ldots(2p-(p-1))}{1\cdot 2 \ldots (p-1)}  \\
   & = (-1)^{p-1}\left(1-\frac{2p}{1}\right)\left(1-\frac{2p}{2}\right)\ldots \left(1-\frac{2p}{p-1}\right).
   \end{align*}
Expanding  the last line in $\Z_{p^5}$ gives us:
\begin{equation}\label{d}
1 - 2p\sum_{ i }\frac{1}{i} + 4p^2\sum_{i < j}\frac{1}{ij} - 8p^3\sum_{i<j<k}\frac{1}{ijk} + 16p^4\sum_{i<j<k<l }\frac{1}{ijkl},
\end{equation}
where here and below, unless otherwise stated, the summations are over variables in the range $1,...,p-1$. 
Next we work ``upwards'':
\begin{align*}
 \binom{2p-1}{p-1}   &= \frac{(1+p)(2+p)\ldots((p-1)+p)}{1\cdot 2 \ldots (p-1)}  \\
   & = \left(1+\frac{p}{1}\right)\left(1+\frac{p}{2}\right)\ldots \left(1+\frac{p}{p-1}\right)
\end{align*}
to obtain in $\Z_{p^5}$ :
\begin{equation}\label{u}
1 + p\sum_{i }\frac{1}{i} + p^2\sum_{ i < j }\frac{1}{ij} + p^3\sum_{i<j<k}\frac{1}{ijk} + p^4\sum_{ i<j<k<l }\frac{1}{ijkl}.
\end{equation}
Multiply equation (\ref{u}) by 2 and add the product to equation (\ref{d}) in order to eliminate the $p$ term. Then divide both members by 3 to get:
\begin{equation*}
\binom{2p-1}{p-1}\equiv1 + 2p^2\sum_{ i < j}\frac{1}{ij} - 2p^3\sum_{i<j<k}\frac{1}{ijk} + 6p^4\sum_{ i<j<k<l }\frac{1}{ijkl}.
\end{equation*}
Concerning the last summand, notice that multiplying all the indices $i,j,k,l$ by 2 leaves the sum $\sum_{ i<j<k<l }\frac{1}{ijkl}$  fixed in $\Z_p$. Therefore,  since $2^4\not\equiv 0\pmod p$, this sum is equivalent to $0\pmod{p}$. The second summand may be transformed by using $2\sum \frac{1}{ij}=\left(\sum \frac{1}{i}\right)^2-\sum \frac{1}{i^2}$. After substitution and application of (\ref{e2}) to the square term we get:
\begin{equation*}
\binom{2p-1}{p-1}\equiv1 -p^2\sum_{i  }\frac{1}{i^2} - 2p^3\sum_{ i<j<k}\frac{1}{ijk} \pmod{p^5}.
\end{equation*}
Finally, concerning the last summand, notice that we have:
\begin{equation*}
6\sum_{ i <j <k }\frac{1}{ijk}= \left(\sum_{i }\frac{1}{i}\right)^3-3\left(\sum_{ i }\frac{1}{i^2}\right)\left(\sum_{ j}\frac{1}{j}\right) + 2\sum_{ i }\frac{1}{i^3}
\end{equation*}
  which is equivalent to $0 \pmod{p^2}$ by using  the equivalences (\ref{e2}), (\ref{e1}) and (\ref{ig2}). Therefore we have proved
\begin{equation}
\label{e4}
\binom{2p-1}{p-1}\equiv1 -p^2\sum_{i }\frac{1}{i^2} \pmod{p^5},
\end{equation}
which figures in \cite[p.~385]{Mc}.

(b) $\Leftrightarrow$ (c). By using elementary identities we obtain:
\begin{align*}
2 \sum_{i}\frac{1}{i}    &=  \sum_{i}\left(\frac{1}{p-i}+\frac{1}{i} \right) =  p\sum_{i}\left(\frac{1}{(p-i)i}+\frac{1}{i^2} -\frac{1}{i^2} \right) \\
 &=   -p\sum_{i}\frac{1}{i^2} +p^2\sum_{i}\frac{1}{(p-i)i^2} \\
   &= -p\sum_{i}\frac{1}{i^2} +  p^2\sum_{i}\left(\frac{1}{(p-i)i^2}+\frac{1}{i^3}\right)-p^2\sum_{i}\frac{1}{i^3} \\ 
   &= -p\sum_{i}\frac{1}{i^2} - p^2\sum_{i}\frac{1}{i^3} +  p^3\sum_{i}\frac{1}{(p-i)i^3},
\end{align*}
from which we easily conclude by using (\ref{ig2}) on the middle summand  and (\ref{ig1}) on the last summand as $\sum_{i}\frac{1}{(p-i)i^3}\equiv \sum_{i}\frac{-1}{i^4}\pmod p$.

Equation (\ref{e4}) $\Leftrightarrow$ (d).
This last equivalence requires basic knowledge of Bernoulli numbers we recall from \cite{IR}. If 
\begin{equation}
\label{Sm}
S_m(p):=\sum_{i=1}^{p-1}i^m
\end{equation}
then from \cite[pg. 230, Theorem 1]{IR}, 
\begin{equation}
\label{ SB}
S_m(p)=\sum_{i=1}^{m+1}\frac{1}{i}\binom{m}{i-1}p^iB_{m+1-i}.
\end{equation}
Importantly for us, the $B_i$ are 0 for odd integers $i>1$. Our general method is to transform the summand in (\ref{e4}) into an equation of the form (\ref{Sm}) by applying Euler's theorem, 
\[
i^{-2}\equiv i^{\phi(p^3)-2} \pmod{p^3},\] where $\phi$ is Euler's totient function. Working in $\Z_{p^3}$ and letting $m:=p^3-p^2-2$ we get, since  odd indexed Bernoulli numbers vanish,
\begin{align*}
\sum_{i=1}^{p-1}i^{-2}\equiv & \sum_{i=1}^{p-1} i^{m} =S_m(p) = \sum_{i=1}^{m+1}\frac{1}{i}\binom{m}{i-1}p^iB_{m+1-i}  \\
& \equiv pB_{p^3-p^2-2}\pmod{p^3}
\end{align*}
which replaced in (\ref{e4}) gives what is wanted:
\[
\binom{2p-1}{p-1}\equiv1 -p^3B_{p^3-p^2-2} \pmod{p^5}.
\]
\end{proof}

McIntosh commented that there is probably only a finite number of primes verifying Theorem \ref{T:mod5}(a) and conjectured that there are none \cite[bottom p.~387]{Mc}. One natural question is: Can Theorem \ref{T:mod5} be extended to the next degree? According to \cite[Lemma 3 and Cor.~5(1)]{HT} it seems the answer is no, since they obtain the following results:
\begin{align*}
  \binom{2p-1}{p-1}  &\equiv 1 - p^3B_{p^3-p^2-2}+\frac{1}{3}p^5B_{p-3}-\frac{6}{5}p^5B_{p-5}  \pmod{p^6}, \\
    \sum_{i=1}^{p-1}\frac{1}{i} & \equiv -\frac{p^2}{2}B_{p^3-p^2-2}+\frac{p^4}{6}B_{p-3}-\frac{p^4}{5}B_{p-5}  \pmod{p^5}, \\
    \sum_{i=1}^{p-1}\frac{1}{i^2} & \equiv pB_{p^3-p^2-2}-\frac{p^3}{3}B_{p-3}+\frac{4}{5}p^3B_{p-5}  \pmod{p^4},
\end{align*} 
and so the last term in $B_{p-5}$ does not coincide in any pair of expressions. Notice that reducing these three equivalences modulo $p^5,p^4,p^3$ respectively establishes Theorem \ref{T:mod5}. It is in this sense that Theorem \ref{T:mod5} is contained in \cite{HT}. A formula for  $\binom{2p-1}{p-1} $ modulo $p^7$ is given in \cite{MeRM}. For related results see \cite{Me}.

\providecommand{\bysame}{\leavevmode\hbox to3em{\hrulefill}\thinspace}
\providecommand{\MR}{\relax\ifhmode\unskip\space\fi MR }
\providecommand{\MRhref}[2]{%
  \href{http://www.ams.org/mathscinet-getitem?mr=#1}{#2}
}
\providecommand{\href}[2]{#2}

\end{document}